\documentclass[reqno,a4paper, 11pt]{amsart}

\usepackage[a4paper=true,pdfpagelabels]{hyperref}
\usepackage{graphicx}

\usepackage[ansinew]{inputenc}
\usepackage{amsfonts,epsfig}
\usepackage{latexsym}
\usepackage{amsmath}
\usepackage{amssymb}

\usepackage{color}

\newtheorem{theorem}{Theorem}
\newtheorem{lemma}[theorem]{Lemma}

\newtheorem{proposition}[theorem]{Proposition}

\theoremstyle{definition}

\theoremstyle{remark}

\numberwithin{equation}{section}

\newcommand{\D}{\mathbb{D}}
\newcommand{\DD}{\widehat{\mathcal{D}}}
\newcommand{\Dd}{\check{\mathcal{D}}}
\newcommand{\M}{\mathcal{M}}
\newcommand{\DDD}{\mathcal{D}}
\newcommand{\De}{\mathcal{D}}
\newcommand{\N}{\mathbb{N}}

\newcommand{\C}{\mathbb{C}}

\renewcommand{\phi}{\varphi}

       \def\De{{\Delta}}    
     \def\om{\omega}      
              
                  \def\z{\zeta}
                  \def\vp{\varphi}

\renewcommand{\H}{\mathcal{H}}

\begin{document}

\title[GENERALIZED WEIGHTED COMPOSITION OPERATORS ON BERGMAN SPACES]{GENERALIZED WEIGHTED COMPOSITION OPERATORS ON BERGMAN SPACES INDUCED BY DOUBLING WEIGHTS}

\author{Bin Liu}
\address{University of Eastern Finland, P.O.Box 111, 80101 Joensuu, Finland}
\email{binl@uef.fi}

\thanks{This research was supported in part by China Scholarship Council No. 201706330108.}

\begin{abstract}
Bounded and compact generalized weighted composition operators acting from the weighted Bergman space $A^p_\omega$, where $0<p<\infty$ and $\omega$ belongs to the class $\DDD$ of radial weights satisfying a two-sided doubling condition, to a Lebesgue space $L^q_\nu$ are characterized. On the way to the proofs a new embedding theorem on weighted Bergman spaces $A^p_\omega$ is established. This last-mentioned result generalizes the well-known characterization of the boundedness of the differentiation operator $D^n(f)=f^{(n)}$ from the classical weighted Bergman space $A^p_\alpha$ to the Lebesgue space $L^q_\mu$, induced by a positive Borel measure $\mu$, to the setting of doubling weights. 
\end{abstract}

\maketitle

\section{INTRODUCTION AND MAIN RESULTS}

Let $\H(\D)$ denote the space of analytic functions in the unit disc $\D=\{z\in\C:|z|<1\}$. An integrable function $\om:\D\to[0,\infty)$ is a weight. It is radial if $\om(z)=\om(|z|)$ for all $z\in\D$. For $0<p<\infty$ and a weight $\omega$, the weighted Bergman space $A^p_\omega$ consists of $f\in\H(\D)$ such that
    $$
    \|f\|_{A^p_\omega}^p=\int_\D|f(z)|^p\omega(z)\,dA(z)<\infty,
    $$
where $dA(z)=\frac{dx\,dy}{\pi}$ is the normalized
Lebesgue area measure on $\D$. The corresponding Lebesgue space is denoted by $L^p_\om$, and thus $A^p_\om=L^p_\om\cap\H(\D)$. As usual,~$A^p_\alpha$ stands for the classical weighted
Bergman space induced by the standard radial weight $\omega(z)=(1-|z|^2)^\alpha$, where
$-1<\alpha<\infty$. For $0<p\le\infty$, the Hardy space $H^p$ consists of functions $f\in\H(\D)$ such that 
     $$
     \|f\|_{H^p}=\sup_{0<r<1} M_p(r,f)< \infty,
     $$
where
     $$
     M_p(r,f)=\left(\frac{1}{2\pi}\int_0^{2\pi}|f(re^{i\theta})|^p d\theta\right)^{\frac{1}{p}}, \quad 0<p<\infty,
     $$
are the $L^p$-means and $M_\infty(r,f)=\max_{|z|=r}|f(z)|$ is the maximum modulus function.

For a radial weight $\om$, write $\widehat{\om}(z)=\int_{|z|}^1\om(s)\,ds$ for all $z\in\D$. A weight $\om$ belongs to the class~$\DD$ if
there exists a constant $C=C(\om)\ge1$ such that $\widehat{\om}(r)\le C\widehat{\om}(\frac{1+r}{2})$ for all $0\le r<1$.
Moreover, if there exist $K=K(\om)>1$ and $C=C(\om)>1$ such that $\widehat{\om}(r)\ge C\widehat{\om}\left(1-\frac{1-r}{K}\right)$ for all $0\le r<1$, then we write $\om\in\Dd$. Weights $\om$ belonging to $\DDD=\DD\cap\Dd$ are called doubling.
If there exist $C=C(\om)>1$ and $K=K(\om)>1$ such that $\om_{x}\ge C\om_{Kx}$ for all $x\ge1$, we write $\om\in\M$.

Let $\vp$ be an analytic self-map of $\D$. The function $\varphi$ induces a composition operator $C_\varphi$ on $\H(\D)$ defined by $C_\varphi f=f\circ\varphi$. The weighted composition operator induced by a self-map $\vp$ and $u\in\H(\D)$ is the operator $u C_\vp$ that sends $f$ to the analytic function  $u\cdot f\circ\varphi$. These operators have been extensively studied in a variety of function spaces, see~for~example \cite{CCM,ZR,GTE,SJH1,SJH2,SW,ZK}. 

The differentiation operator $D$ on $\H(\D)$ is defined by $Df=f'$. Further, for $n\in \N$, we define $D^nf=f^{(n)}$. By following~\cite{ZX}, the generalized weighted composition operator $D^n_{\varphi,u}$ is defined by 
\begin{equation}
D^n_{\varphi,u}f=u\cdot f^{(n)}\circ\varphi,
\end{equation}
where $\varphi$ is an analytic self-map of $\D$, $u\in\H(\D)$ and $n\in\N$. Obviously, if $n=0$ and $u\equiv1$, then this operator reduces to the composition operator $C_\varphi$, and if $n=0$, then we get the weighted composition operator $uC_\varphi$. If $n=1$ and $u(z)=\varphi'(z)$, then $D^n_{\varphi,u}=DC_\varphi$ which was studied in \cite{HR} \cite{LS} . When $n=1$ and $u\equiv1$, then $D^n_{\varphi,u}=C_\varphi D$ which was studied in \cite{HR}.

In \cite{ZX1}, the author characterized the boundedness and compactness of the operators $D^n_{\varphi,u}$ between different standard weighted Bergman spaces. In this paper, we characterize bounded and compact operators $D^n_{\varphi,u}$ acting from the weighted Bergman space $A^p_\omega$ with $\omega \in \DDD$ to $L^q_\nu$ induced by any $0<q<\infty$ and a positive Borel measure $\nu$.   

To state main results, some more notation is needed. Let $\mu$ be a finite positive Borel measure on $\D$ and $h$ a measurable function on $\D$. For an analytic self-map $\varphi$ of $\D$, the weighted pushforward measure related to $h$ is defined by 
	\begin{equation}\label{Eq:pusforward}
	\varphi_*(h,\mu)(M)=\int_{\varphi^{-1}(M)}h d\mu
	\end{equation}
for each measurable set $M\subset\D$. If $\mu$ is the Lebesgue measure, we omit the measure in the notation and write $\varphi_*(h)(M)$ for the left hand side of \eqref{Eq:pusforward}. Here and from now on $S(z)=\{\zeta\in\D:1-|z|<|\zeta|,\,|\arg \zeta-\arg z|<(1-|z|)/2\}$ is the Carleson square induced by the point $z\in\D\setminus\{0\}$, $S(0)=\D$ and $\om(E)=\int_E\om dA$ for each measurable set $E\subset\D$. Further, $\Gamma(z)=\left\{\z\in\D:|\theta-\arg \z|<\frac{1}{2}\left(1-\frac{|\z|}{r}\right)\right\}$, $z=re^{i\theta}\in\overline{\D}\setminus\{0\}$ is a non-tangential approach region with vertex at $z\in\overline{\D}\setminus\{0\}$, and $T(z)=\{\z\in\D:z\in\Gamma(\z)\}$ is the tent induced by $z\in\D\setminus\{0\}$. Observe that we have $\om(T(z))\asymp\om(S(z))$ for all $z\in\D\setminus\{0\}$ if $\om\in\DD$. The statement of our first result involves pseudohyperbolic discs. The pseudohyperbolic distance between two points $a$ and $b$ in $\D$ is $\rho(a,b)=|(a-b)/(1-\overline{a}b)|$. For $a\in\D$ and $0<r<1$, the pseudohyperbolic disc of center $a$ and of radius $r$ is $\Delta(a,r)=\{z\in \D:\rho(a,z)<r\}$. It is well known that $\Delta(a,r)$ is an Euclidean disk centered at $(1-r^2)a/(1-r^2|a|^2)$ and of radius $(1-|a|^2)r/(1-r^2|a|^2)$. Finally, denote $\widetilde{\omega}(z)=\widehat{\omega}(z)/(1-|z|)$ for all $z\in\D$. By \cite[Proposition~5]{JJK} we know that
	\begin{equation}\label{eq:equivalent norms}
	\|f\|_{A^p_{\widetilde{\om}}}\asymp\|f\|_{A^p_\om},\quad f\in\H(\D),
	\end{equation}
provided $\om\in\DDD$.

Our first result reads as follows:

\begin{theorem} \label{Carleson DD}
Let $0<q<p<\infty$, $\omega\in \DDD$, $n \in \N \cup\{0\}$, and let $\nu$ be a positive Borel measure on $\D$. Let $\varphi$ be an analytic self-map of $\D$ and $u\in L^q_\nu$. Then the following
statements are equivalent:
  \begin{itemize}
   \item[\rm(i)] $D^n_{\varphi,u}: A^p_\omega \rightarrow L^q_\nu$ is bounded;
   \item[\rm(ii)] $D^n_{\varphi,u}: A^p_\omega \rightarrow L^q_\nu$ is compact;
   \item[\rm(iii)]$\frac{\varphi_*(|u|^q \nu)(\Delta(z,r))}{\omega(S(z))(1-|z|)^{nq}}$ belongs to $L^{\frac{p}{p-q}}_{\widetilde{\omega}}$ 
   for some (equivalently for all) $r \in (0,1)$.
  \end{itemize}
\end{theorem}

We will need one specific tool for the proof of Theorem~\ref{Carleson DD}. To put it in a bigger whole, couple of words on continuous embeddings $A^p_\om\subset L^q_\mu$ are in order. For a positive Borel measure $\mu$ on $\D$ and $0<\alpha<\infty$, define the weighted maximal function 
	\begin{equation*}
	M_{\omega,\alpha}(\mu)(z)=\sup_{z\in S(a)}\frac{\mu(S(a))}{(\omega(S(a)))^\alpha},\quad z\in \D,
	\end{equation*}
and write $M_\omega(\mu)=M_{\omega,1}(\mu)$ for short. If the identity operator $I:A^p_\om\to L^q_\mu$ is bounded, then $\mu$ is called a $q$-Carleson measure for $A^p_\om$. A complete characterization of such measures in the case $\om\in\DD$ can be found in \cite{PelRatEmb}, see also \cite{PR,PelSum14}. In particular, it is known that if $q\ge p$ and $\om\in\DD$, then $\mu$ is a $q$-Carleson measure for $A^p_\om$ if and only if $M_{\omega,q/p}(\mu) \in L^\infty$, and if $p>q$, then $\mu$ is a $q$-Carleson measure for $A^p_\om$ if and only if $M_\omega(\mu) \in L^{\frac{p}{p-q}}_\omega$. The standard Littlewood-Paley formula implies that the bounded differentiation operators $D^n f=f^{(n)}$ from $A^p_\alpha$ to $L^q_\mu$ can be characterized once the $q$-Carleson measures for $A^p_\alpha$ are characterized. A characterization of such operators on the weighted Bergman spaces $A^p_\omega$ with $\omega\in\DD$ can be found in \cite{PelRatEmb}. In this paper, we are interested in the case of $\omega\in\DDD$ and this is what we need for the proof of Theorem~\ref{Carleson DD}. The following theorem is our second main result and generalizes \cite[Theorem~2.2]{LD1} and \cite[Theorem~1]{LD} to doubling weights.

\begin{theorem}\label{Carlson D}
Let $0<p, q< \infty$ , $\omega \in \DDD$ and $n \in \N \cup\{0\}$, and let $\mu$ be a positive Borel measure on $\D$. 
\begin{itemize}
\item[(a)]If $0<p\leq q< \infty$, then the following statements hold:
  \begin{itemize}
   \item[(i)] $D^{n}: A^p_\omega \rightarrow L^q_\mu$ is bounded if and only if 

     \begin{equation} \label{eq:bounded}
     \sup_{z\in \D} \frac{\mu(\Delta(z,r))}{\omega(S(z))^{\frac{q}{p}}(1-|z|)^{nq}}< \infty.
     \end{equation}
   \item[(ii)] $D^{n}: A^p_\omega \rightarrow L^q_\mu$ is compact if and only if 

       \begin{equation} \label{eq:compact}
       \lim_{|z|\rightarrow 1^-} \frac{\mu(\Delta(z,r))}{\omega(S(z))^{\frac{q}{p}}(1-|z|)^{nq}} =0.
       \end{equation}
   \end{itemize}
\item[(b)] If $0<q<p<\infty$, then the following conditions are equivalent:
   \begin{itemize}
     \item[(i)] $D^{n}: A^p_\omega \rightarrow L^q_\mu$ is bounded;
     \item[(ii)] $D^{n}: A^p_\omega \rightarrow L^q_\mu$ is compact;
     \item[(iii)] the funcion 
			$$
			z\mapsto\frac{\mu(\Delta(z,r))}{\omega(S(z))(1-|z|)^{nq}}
			$$ 
     belongs to $L^{\frac{p}{p-q}}_{\widetilde{\omega}}$ for some (equivalently for all) 
      $r\in (0,1)$.
   \end{itemize}
\end{itemize}
\end{theorem}

When $n=0$, Parts (a) and (b) of Theorem~\ref{Carlson D} convert to \cite[Theorem 2]{LR2020} and \cite[Theorem 2]{LRW2020}, respectively. If $n\in\N$, the method of proof of \cite[Theorem 2]{LR2020} gives (a) immediately. Therefore our contribution consists of proving (b) for $n\in\N$ by establishing the implications (ii)$\Rightarrow$ (i)$\Rightarrow$ (iii)$\Rightarrow$ (ii). 

The next main result of this work concerns the boundedness and compactness of $D^n_{\varphi,u}$ when $p\leq q$ and $\omega \in \DDD$.

\begin{theorem} \label{theorem1}
Let $0<p\leq q< \infty$ and $\omega, \nu\in \DDD$. Let $\varphi$ be an analytic self-map of $\D$, $u\in A^q_\nu$, and $n \in \N \cup\{0\}$. Then there exists $\gamma=\gamma(p,\omega)>0$ such that the following statements hold:
 \begin{itemize}
 \item[\rm(i)]$D^n_{\varphi,u}: A^p_\omega \rightarrow A^q_\nu$ is bounded if and only if
   \begin{equation}\label{sup}
   \sup_{a\in \D} \int_\D|u(z)|^q\frac{(1-|a|)^{\gamma q}}{|1-\overline{a}\varphi(z)|^{(\gamma+n) q}}\frac{\nu(z)}
   {\omega(S(a))^{\frac{q}{p}}}dA(z)< \infty; 
   \end{equation}
 \item[\rm(ii)] $D^n_{\varphi,u}: A^p_\omega \rightarrow A^q_\nu$ is compact if and only if
   \begin{equation} \label{lim}
   \lim_{|a|\rightarrow 1} \int_\D|u(z)|^q\frac{(1-|a|)^{\gamma q}}{|1-\overline{a}\varphi(z)|^{(\gamma+n) q}}\frac{\nu(z)}
   {\omega(S(a))^{\frac{q}{p}}}dA(z)=0. 
   \end{equation}
 \end{itemize}
\end{theorem}

If $\om(z)=(1-|z|)^{\alpha}$ and $\nu(z)=(1-|z|)^\beta$ with $-1<\alpha,\beta<\infty$, then $\widehat{\om}(z)\asymp(1-|z|)^{\alpha+1}$ and $\widehat{\nu}(z)\asymp(1-|z|)^{\beta+1}$ for all $z\in\D$. The proof of Theorem~\ref{theorem1} shows that the only requirement for $\gamma=\gamma(\om,p)>0$ appearing in the statement is that 
	$$
	\int_\D\frac{\om(z)}{|1-\overline{a}z|^{\gamma p}}\,dA(z)\le C\frac{\widehat{\om}(a)}{(1-|a|)^{\gamma p-1}},\quad a\in\D,
	$$
for some constant $C=C(\om,p,\gamma)>0$. If $\om(z)=(1-|z|)^{\alpha}$, any $\gamma>\frac{\alpha+2}{p}$ is acceptable, and the choice $2\frac{\alpha+2}{p}$ converts Parts (i) and (ii) of Theorem~\ref{theorem1} to Theorems 2.6 and 2.7 in \cite{ZX1} respectively, as a simple computation shows. Therefore Theorem~\ref{theorem1} indeed generalizes \cite[Theorem 2.6 and Theorem 2.7]{ZX1} for weights in $\DDD$. 

Our last main result concerns the case in which the target space is $H^\infty$.

\begin{theorem} \label{Th}
Let $\omega \in \DDD$, $0<p<\infty$ and $n \in \N \cup\{0\}$. Let $\varphi$ be an analytic self-map of $\D$ and $u\in H^\infty(\D)$. Then $D^n_{\varphi,u}: A^p_\omega \rightarrow H^\infty$ is bounded if and only if

\begin{equation} \label{Dd}
\sup_{z\in \D} \frac{|u(z)|}{\omega(S(\varphi(z)))^{\frac{1}{p}}(1-|\varphi(z)|)^n}< \infty.
\end{equation}
Moreover, if $D^n_{\varphi,u}: A^p_\omega \rightarrow H^\infty$ is bounded, then 

\begin{equation}\label{futhermore}
\|D^n_{\varphi,u}\|\asymp \sup_{z\in \D} \frac{|u(z)|}{\omega(S(\varphi(z)))^{\frac{1}{p}}(1-|\varphi(z)|)^n}
\end{equation}

Furthermore, with the same assumptions the following are also equivalent:

\begin{itemize}
\item[(i)] $D^n_{\varphi,u}: A^p_\omega \rightarrow H^\infty$ is compact;

\item[(ii)] $\sup_{z\in \D}|\varphi(z)|<1$ or 

  \begin{equation} \label{condition2}
   \lim_{|\varphi(z)|\rightarrow 1} \frac{|u(z)|}{\omega(S(\varphi(z)))^{\frac{1}{p}}(1-|\varphi(z)|)^n}=0.
  \end{equation}
\end{itemize}
\end{theorem}
If $\om(z)=(1-|z|)^{\alpha}$ for some $-1<\alpha<\infty$, then $\om(S(\zeta))\asymp \om(\zeta)(1-|\zeta|)^2= (1-|\zeta|)^{2+\alpha}$ for all $\zeta\in\D$. Therefore our results generalize \cite[Theorems~2.9 and 2.10]{ZX1} for weights in $\DDD$.

The rest of the paper contains the proofs of the results stated above. In the next section we go through auxiliary results. The proof of Theorem~\ref{Carlson D} is given in Section~\ref{Sec:embedding}. In Sections~\ref{Sec:boundedness} and \ref{compactness}, we characterize the boundedness and compactness of the operator $D^n_{\varphi,u}$. Finally, we show how to get the main results in Section~\ref{Sec:proof-of-main}.

Before going further couple of words about the notation used. The letter $C=C(\cdot)$ will denote an absolute constant whose value depends on the parameters indicated
in the parenthesis, and may change from one occurrence to another.
We will use the notation $a\lesssim b$ if there exists a constant
$C=C(\cdot)>0$ such that $a\le Cb$, and $a\gtrsim b$ is understood
in an analogous manner. In particular, if $a\lesssim b$ and
$a\gtrsim b$, then we write $a\asymp b$ and say that $a$ and $b$ are comparable.

\section{PRELIMINARIES}
It is known that if $\om\in\DDD$, then there exist constants $0<\alpha=\alpha(\om)\le\beta(\om)<\infty$ and $C=C(\om)\ge1$ such that
	
	\begin{equation}\label{Eq:characterization-D}
	\frac1C\left(\frac{1-r}{1-t}\right)^\alpha
	\le\frac{\widehat{\om}(r)}{\widehat{\om}(t)}
	\le C\left(\frac{1-r}{1-t}\right)^\beta,\quad 0\le r\le t<1.
	\end{equation}
In fact, this pair of inequalities characterizes the class $\DDD$ because the right hand inequality is satisfied if and only if $\om\in\DD$ by \cite[Lemma~2.1]{PelSum14} while the left hand inequality describes the class $\Dd$ in an analogous way, see \cite[(2.27)]{PelRat2020}. The chain of inequalities \eqref{Eq:characterization-D} will be frequently used in the sequel.

Another result needed is a lemma that allows us to estimate $|f^{(n)}(z)|$ sufficiently accurately when $f\in A^p_\om$ with $\om\in\DDD$. 

\begin{lemma}\label{subharmonic}
Let $\omega \in \DDD$, $0<p<\infty$ and $n\in \N \cup \{0\}$. If $f\in A^p_\omega$, then we have
\begin{equation}
|f^{(n)}(z)|\leq C\frac{\|f\|_{A^p_\omega}}{(\omega(S(z)))^{\frac{1}{p}}(1-|z|)^n},\quad z\in \D.
\end{equation}
\end{lemma}

\begin{proof}
By \eqref{eq:equivalent norms}, $1-|\xi|\asymp 1-|z|$ for all $\xi \in \Delta(z,r)$ and $\widehat{\omega}(z)(1-|z|)\asymp \omega(S(z))$ for all $z\in\D$. Therefore 
\begin{equation*}
 \begin{split}
  |f^{(n)}(z)|^p 
  &\leq\frac{C}{(1-|z|)^2} \int_{\Delta(z,r)} |f^{(n)}(\xi)|^p dA(\xi)\\
  &\leq \frac{C}{(1-|z|)^{2+np}} \int_{\Delta(z,r)} |f^{(n)}(\xi)|^p (1-|\xi|)^{np} dA(\xi)\\
  &\asymp \frac{1}{\widehat{\omega}(z)(1-|z|)^{1+np}} \int_{\Delta(z,r)} |f^{(n)}(\xi)|^p (1-|\xi|)^{np} \widetilde{\omega}(\xi) dA(\xi)\\
  &\lesssim \frac{1}{\omega(S(z))(1-|z|)^{np}} \int_\D |f^{(n)}(\xi)|^p (1-|\xi|)^{np} \omega(\xi) dA(\xi),
 \end{split}
\end{equation*}
and hence\cite[Theorem 5]{PelRat2020} yields 
\begin{equation*}
|f^{(n)}(z)|\lesssim\frac{\|f\|_{A^p_\omega}}{(\omega(S(z)))^{\frac{1}{p}}(1-|z|)^n},\quad z\in\D.
\end{equation*}
The lemma is proved.
\end{proof}

The next lemma follows by standard arguments, see, for example,~\cite[Proposition 3.11]{CCM}. We omit the details of the proof.

\begin{lemma}\label{compact}
Let $\varphi$ be an analytic self-map of $\D$, u $\in \H(\D)$ and $n\in\N\cup\{0\}$. Let $0 < p,q < \infty$, and $\omega,\nu$ in $\DDD$. Then 
$D^n_{\varphi,u}: A^p_\omega \rightarrow A^q_\nu$ is compact 
if and only if $D^n_{\varphi,u}: A^p_\omega \rightarrow A^q_\nu$ is bounded and
for any bounded sequence $(f_k)_{k\in \N}$ in $A^p_\omega$, which converges to zero uniformly on
compact subsets of $\D$, we have $\|D^n_{\varphi,u} f_k\|_{A^q_\nu}\rightarrow 0$ as $k\rightarrow \infty$.
\end{lemma}

\section{EMBEDDING THEOREMS} \label{Sec:embedding}

In this section we establish embeddings theorem of  $A^p_\omega$ into $L^q_\mu$ with $0<p,q<\infty$ and $\omega\in \DDD$. We begin with the case $0<p\leq q<\infty$. 

\begin{proposition}\label{Carleson d}
Let $0<p\leq q< \infty$ , $\omega \in \DDD$ and $n \in \N \cup\{0\}$, and let $\mu$ be a positive Borel measure on $\D$. Then there exists $r=r(\omega)\in(0,1)$ such that the following statements hold:

   \begin{itemize}
   \item[(i)] $D^{n}: A^p_\omega \rightarrow L^q_\mu$ is bounded if and only if 

     \begin{equation} \label{eq:bounded}
     \sup_{z\in \D} \frac{\mu(\Delta(z,r))}{\omega(S(z))^{\frac{q}{p}}(1-|z|)^{nq}}< \infty.
     \end{equation}
   \item[(ii)] $D^{n}: A^p_\omega \rightarrow L^q_\mu$ is compact if and only if 

       \begin{equation} \label{eq:compact}
       \lim_{|z|\rightarrow 1^-} \frac{\mu(\Delta(z,r))}{\omega(S(z))^{\frac{q}{p}}(1-|z|)^{nq}} =0.
       \end{equation}
   \end{itemize}
\end{proposition}

\begin{proof}
When $n=0$, the statement is proved in \cite[Theorem~2]{LR2020}. Hence we may assume $n\in\N$. The necessity of the condition in (i) can be proved easily. For $a\in \D$, define the function  
    \begin{equation} \label{eq:testfunctions}
    f_a(z)=\left(\frac{1-|a|}{1-\overline{a}z}\right)^\gamma\omega(S(a))^{-\frac{1}{p}}, \quad z\in \D,
    \end{equation}
induced by $\omega$ and $0< \gamma,p< \infty$. Then \cite[Lemma 2.1]{PelSum14} implies that for all $\gamma=\gamma(\omega,p)>0$ sufficiently large we have  
$\|f_a\|_{A^p_\omega}\asymp 1$ for all $a\in \D$. By using \cite[Lemma 8]{PelRatEmb} it is easy to see that \eqref{eq:bounded} holds. 

To prove the sufficiency, we assume that \eqref{eq:bounded} holds. Recall the known estimate
    \begin{equation}
    |f^{(n)}(z)|^p \lesssim \frac{1}{(1-|z|)^{2+np}}\int_{\Delta(z,r)}|f(\xi)|^p dA(\xi), \quad z\in\D,
    \end{equation}
see, for example, \cite[Lemma 2.1]{LD1} for details. This estimate, Minkowski's inequality in continuous form (Fubini's theorem in the case $q = p$) and \eqref{eq:bounded} give
    $$
\|f^{(n)}\|_{L^q_\mu}^q \lesssim\left(\int_\D|f(\zeta)|^p\frac{\om(S(\zeta))}{(1-|\zeta|)^2}\,dA(\zeta)\right)^\frac{q}{p},\quad f\in\H(\D),
	$$
see \cite[Theorem 2 (i)]{LR2020} for details. Since $\om\in\DDD$ by the hypothesis, we may apply the right hand inequality in \eqref{Eq:characterization-D} to deduce 
   \begin{equation}\label{1111}
	\om(S(\zeta))\lesssim\widehat{\om}(\zeta)(1-|\zeta|),\quad \zeta\in\D.
	\end{equation}
It follows that $\|f^{(n)}\|_{L^q_\mu}\lesssim\|f\|_{A^p_{\widetilde{\om}}}$, and hence $\|f^{(n)}\|_{L^q_\mu}\lesssim\|f\|_{A^p_{\om}}$ for all $f\in\H(\D)$ by \eqref{eq:equivalent norms}. Thus $D^{n}$ is bounded, and (i) is proved.  

The statement concerning the compactness can be proved by following the proof of \cite[Theorem 2 (ii)]{LR2020} line by line, and therefore
we omit the derails.
\end{proof}

The next proposition deals with the case $q<p$. 

\begin{proposition} \label{Carleson d2}
Let $0<q<p<\infty$, $\omega \in \DDD$ and $n \in \N \cup\{0\}$ and let $\mu$ be a positive Borel measure on $\D$. Then the following conditions are equivalent:
\begin{itemize}
\item[(i)] $D^{n}: A^p_\omega \rightarrow L^q_\mu$ is bounded;
\item[(ii)] $D^{n}: A^p_\omega \rightarrow L^q_\mu$ is compact;
\item[(iii)] $\frac{\mu(\Delta(z,r))}{\omega(S(z))(1-|z|)^{nq}}$ belongs to $L^{\frac{p}{p-q}}_{\widetilde{\omega}}$ for some (equivalently for all) $r\in (0,1)$.
\end{itemize}
\end{proposition}

\begin{proof}
If $n=0$, the statement reduces to \cite[Theorem 2]{LRW2020}, and hence we may assume that $n\in\N$. The implication
(ii) $\Rightarrow$ (i) is obvious. To prove that (i) implies (iii), assume that $D^n:A^p_\om\to L^q_\mu$ is bounded. Let $\{z_k\}$ be a $r$-lattice such that $z_k\neq0$ for all $k$. By \cite[Theorem~1]{JJK2} there exist constants $\lambda=\lambda(p,\om)>1$ and $C=C(p,\om)>0$ such that the function
	$$
	F(z)=\sum_kb_k\left(\frac{1-|z_k|}{1-\overline{z_k}z}\right)^\lambda\frac{1}{(\omega(T(z_k)))^{\frac{1}{p}}},\quad z\in\D,
	$$
belongs to $A^p_{\omega}$ and satisfies $\|F\|_{A^p_\om}\le C\|b\|_{\ell^p}$ for all $b=\{b_k\}\in\ell^p$. Since $D^n : A^p_\om \rightarrow L^q_\mu$ is bounded by the hypothesis, we deduce
	\begin{equation*}
	\begin{split}
	\|b\|_{\ell^p}^q
	&\gtrsim\|F\|_{A^p_{\om}}^q\gtrsim \int_{\D}\left|\sum_kb_k\left(\frac{(1-|z_k|)^\lambda}{(1-\overline{z_k}z)^{\lambda+n}}\right)\frac{1}{(\omega(T(z_k)))^{\frac{1}{p}}}\right|^qd\mu(z),\quad b\in\ell^p.
	\end{split}
	\end{equation*}
One may now complete this part of the proof by using Khinchine's inequality and properties of an $r$-lattice. Namely, the argument used in  \cite[Proposition 7]{LRW2020} now shows that $\frac{\mu(\Delta(z,r))}{\omega(S(z))(1-|z|)^{nq}}$ belongs to $L^{\frac{p}{p-q}}_{\widetilde{\omega}}$ for some (equivalently for all) $r\in (0,1)$.

It remains to show that (iii) implies (ii). It suffices to show that for any bounded sequence $\{f_k\}$ in $A^p_\omega$ which tends to zero uniformly on compact subsets of $\D$ as $k\rightarrow \infty$, we have $\|D^nf_k\|^q_{L^q_\mu} \rightarrow 0$ as $k\rightarrow \infty$. For simplicity, assume that $\|f_k\|_{A^p_\omega} \leq 1$ for all $n$. By \cite[Lemma 2.1]{LD1}, we have 
\begin{equation} \label{f}
|f^{(n)}(z)|^q \leq \frac{C}{\omega(S(z))(1-|z|)^{nq}}\int_{\Delta(z,r)} |f(\xi)|^q \widetilde{\omega}(\xi)dA(\xi).
\end{equation}
By Fubini's theorem, \eqref{f} and \eqref{Eq:characterization-D}, we deduce 
    $$
\|D^n f_k\|^q_{L^q_\mu} \lesssim \int_\D|f_k(\xi)|^q \frac{\nu(\Delta(\xi,r))}{\omega(S(\xi))(1-|\xi|)^{nq}}\widetilde{\omega}(\xi)dA(\xi).
    $$  
Let $\varepsilon > 0$. The hypothesis implies that there exists an $r \in (0,1)$ such that 
\begin{equation*}
\int_{\D\setminus \overline{D(0,r)}}\left(\frac{\nu(\Delta(\xi,r))}{\omega(S(\xi))(1-|\xi|)^{nq}}\right)^{\frac{p}{p-q}}\widetilde{\omega}(\xi)dA(\xi) \leq \varepsilon^{\frac{p}{p-q}}.
\end{equation*}
By the uniform convergence, we may choose $k_0 \in \N$ such that $|f_k(z)|< \varepsilon^{\frac{1}{q}}$ for all $k\geq k_0$ and $z\in \overline{D(0,r)}$. The H\"older inequality and \eqref{eq:equivalent norms} yield
  \begin{equation*}
   \begin{split}
   \|D^n f_k\|^q_{L^q_\mu} 
   &\lesssim \left(\int_{\overline{D(0,r)}}+ \int_{\D\setminus \overline{D(0,r)}}\right)|f_k(\xi)|^q \frac{\nu(\Delta(\xi,r))}{\omega(S(\xi)) 
   (1-|\xi|)^{nq}}\widetilde{\omega}(\xi)dA(\xi)\\
   &\lesssim \sup_{\xi \in \overline{D(0,r)}} \frac{\nu(\Delta(\xi,r))}{\omega(S(\xi))(1-|\xi|)^{nq}} \int_{\overline{D(0,r)}}|f_k(\xi)|^q  
   \widetilde{\omega}(\xi)dA(\xi)\\
   &\quad + \left(\int_\D |f_k(\xi)|^p \widetilde{\omega}(\xi)dA(\xi)\right)^{\frac{q}{p}} \left(\int_{\D\setminus \overline{D(0,r)}}
   \left(\frac{\nu(\Delta(\xi,r))}{\omega(S(\xi))(1-|\xi|)^{nq}}\right)^{\frac{p}{p-q}}\widetilde{\omega}(\xi)dA(\xi)\right)^{\frac{p-q}{p}}\\
   &\lesssim(1+\|f_k\|^q_{A^p_\omega})\varepsilon \lesssim \varepsilon,
   \end{split}
  \end{equation*}
and thus $D^n : A^p_\omega \rightarrow L^q_\mu$ is compact.
\end{proof}

\section{BOUNDEDNESS} \label{Sec:boundedness}

In this section, we characterize the boundedness of the generalized weighted composition operator $D^n_{\varphi,u}$ on weighted Bergman spaces.

\begin{proposition} \label{bouned pq}
Let $\varphi$ be an analytic self-map of $\D$, $u\in A^q_\nu$, and $n \in \N \cup\{0\}$. Assume $0<p\leq q< \infty$ and $\omega,\nu$ in $\DDD$. Then there exists $\gamma=\gamma(p,\omega)>0$ such that $D^n_{\varphi,u}: A^p_\omega \rightarrow A^q_\nu$ is bounded if and only if

\begin{equation}\label{sup}
 \sup_{a\in \D} \int_\D|u(z)|^q\frac{(1-|a|)^{\gamma q}}{|1-\overline{a}\varphi(z)|^{(\gamma+n) q}}\frac{\nu(z)}
 {\omega(S(a))^{\frac{q}{p}}}dA(z)< \infty.
\end{equation}
\end{proposition}

\begin{proof}
First assume that $D^n_{\varphi,u}: A^p_\omega \rightarrow A^q_\nu$ is bounded. Then there exists a constant $C=C(\gamma,p,n)$ such that
\begin{equation*}
  \begin{split}
  \|D^n_{\varphi,u}f_a\|^q_{A^q_\nu}
  &=(C|a|^n)^q \int_\D|u(z)|^q\frac{(1-|a|)^{\gamma q}}{|1-\overline{a}\varphi(z)|^{(\gamma+n) q}} \frac{\nu(z)}  
  {\omega(S(a))^{\frac{q}{p}}}dA(z)\\
  &\leq C \|f_{a}\|^q_{A^p_\omega}<\infty,
  \end{split}
\end{equation*}
where $f_a$ defined as in \eqref{eq:testfunctions}. It follows that~\eqref{sup} is satisfied.

Conversely, we assume that~\eqref{sup} holds. For each $a\in \D$ and $r>0$, the eatimate $|1-\overline{a}\varphi(z)|\asymp 1-|a|$ for $z\in \varphi^{-1}\De(a,r)$ yields
	\begin{equation*}
  \begin{split}
  \infty 
   &>\int_\D|u(z)|^q\frac{(1-|a|)^{\gamma q}}{|1-\overline{a}\varphi(z)|^{(\gamma+n) q}} \frac{\nu(z)}  
  {\omega(S(a))^{\frac{q}{p}}}dA(z)\\
  &\gtrsim \int_{\varphi^{-1}(\Delta(a,r))} |u(z)|^q\frac{(1-|a|)^{\gamma q}}{|1-\overline{a}\varphi(z)|^{(\gamma+n) q}} \frac{\nu(z)}  
  {\omega(S(a))^{\frac{q}{p}}}dA(z)
	\gtrsim \frac{\varphi_*(|u|^q\nu)(\Delta(a,r))}{\omega(S(a))^{\frac{q}{p}}(1-|a|)^{nq}}, 
  \end{split}
\end{equation*}
and hence the function
\begin{equation*}
 a\mapsto\frac{\varphi_*(|u|^q\nu)(\Delta(a,r))}{\omega(S(a))^{\frac{q}{p}}(1-|a|)^{nq}} 
\end{equation*}
belongs to $L^{\infty}$ for any fixed $r\in (0,1)$. By Proposition~\ref{Carleson d}, we have $\|f^{(n)}\|_{L^q_{\varphi_*(|u|^q\nu)}} \lesssim \|f\|_{A^p_\omega}$ for all $f\in A^p_\omega$. By the measure theoretic change of variable \cite[Section 39]{PM} it follows that 
\begin{equation*}
 \begin{split}
  \|D^n_{\varphi,u}f\|^q_{L^q_\nu}
  &= \int_\D |f^{(n)}(\varphi(z))|^q |u(z)|^q \nu(z) dA(z)\\
  &=\int_\D |f^{(n)}(z)|^q d \varphi_*(|u|^q\nu)(z)\\
  &=\|f^{(n)}\|^q_{L^q_{\varphi_*(|u|^q\nu)}} \lesssim \|f\|^q_{A^p_\omega}.
  \end{split}
\end{equation*}
Thus $D^n_{\varphi,u}: A^p_\omega \rightarrow A^q_\nu$ is bounded. 
\end{proof}

Next, we consider the case $0<p< \infty$ and $q=\infty$. Then a necessary condition for the boundness of the operator $D^n_{\varphi,u}: A^p_\omega \rightarrow H^\infty$ is $u\in H^\infty$. 

\begin{proposition} \label{bounded p}
Let $\omega \in \DDD$, $0<p<\infty$ and $n \in \N \cup\{0\}$. Let $\varphi$ be an analytic self-map of $\D$ and $u\in H^\infty$. Then $D^n_{\varphi,u}: A^p_\omega \rightarrow \H^\infty$ is bounded if and only if
\begin{equation} \label{Dd}
\sup_{z\in \D} \frac{|u(z)|}{\omega(S(\varphi(z)))^{\frac{1}{p}}(1-|\varphi(z)|)^n}< \infty.
\end{equation}
Furthermore, if $D^n_{\varphi,u}: A^p_\omega \rightarrow \H^\infty$ is bounded, then 
	\begin{equation}\label{futhermore}
	\|D^n_{\varphi,u}\|\asymp \sup_{z\in \D} \frac{|u(z)|}{\omega(S(\varphi(z)))^{\frac{1}{p}}(1-|\varphi(z)|)^n}
	\end{equation}
\end{proposition}

\begin{proof}
First, assume that $D^n_{\varphi,u}: A^p_\omega \rightarrow \H^\infty$ is bounded. Consider the function $f_\xi$ with $a=\varphi(\xi)$ defined in \eqref{eq:testfunctions}. Then
\begin{equation*}
\|D^n_{\varphi,u}f_\xi\|_\infty \leq \|D^n_{\varphi,u}\|\|f_\xi\|_{A^p_\omega}\leq C \|D^n_{\varphi,u}\|, 
\end{equation*}
that is, for any $z\in \D$, we have 

\begin{equation*}
|u(z)||f^{(n)}_\xi(\varphi(z))|\leq C \|D^n_{\varphi,u}\|.
\end{equation*}
In particular, letting $z = \xi$, we have

\begin{equation}\label{cpn}
C(\gamma,p,n)\frac{|u(\xi)||\varphi(\xi)|^n}{\omega(S(\varphi(\xi)))^{\frac{1}{p}}(1-|\varphi(\xi)|)^n}\leq C \|D^n_{\varphi,u}\|.
\end{equation}
Therefore
\begin{equation*}
\sup_{\xi\in \D} \frac{|u(\xi)|}{\omega(S(\varphi(\xi)))^{\frac{1}{p}}(1-|\varphi(\xi)|)^n}< \infty.
\end{equation*}

Conversely, assume that \eqref{Dd} holds. Then for $f\in A^p_\omega$, by Lemma~\ref{subharmonic}, we have 
\begin{equation}\label{HD}
  \begin{split}
  |(D^n_{\varphi,u}f)(z)|
  &=|u(z)||f^{(n)}(\varphi(z))|\\
  &\leq C\frac{|u(z)|\|f\|_{A^p_\omega}}{(\omega(S(\varphi(z))))^{\frac{1}{p}}(1-|\varphi(z)|)^n},\quad z\in\D.
  \end{split}
\end{equation}
Thus $D^n_{\varphi,u}: A^p_\omega \rightarrow H^\infty$ is bounded.

By combining \eqref{cpn} and \eqref{HD} we obtain \eqref{futhermore}. This completes the proof.
\end{proof}

\section{COMPACTNESS}\label{compactness}

In this section we will characterize the compactness of the generalized weighted composition operator $D^n_{\varphi,u}$ on weighted Bergman spaces.

\begin{proposition}\label{compact pq}
Let $\varphi$ be an analytic self-map of $\D$, $u\in A^q_\nu$, and $n \in \N \cup\{0\}$. Assume $0<p\leq q< \infty$ and $\omega,\nu \in \DDD$. Then there exists $\gamma=\gamma(p,\omega)>0$ such that $D^n_{\varphi,u}: A^p_\omega \rightarrow A^q_\nu$ is compact if and only if

\begin{equation} \label{lim}
 \lim_{|a|\rightarrow 1} \int_\D|u(z)|^q\frac{(1-|a|)^{\gamma q}}{|1-\overline{a}\varphi(z)|^{(\gamma+n) q}}\frac{\nu(z)}
 {\omega(S(a))^{\frac{q}{p}}}dA(z)=0.
\end{equation}
\end{proposition}

\begin{proof}
First, we assume that $D^n_{\varphi,u}$ is compact. Consider the function $f_a$ defined in \eqref{eq:testfunctions}. It is obviously that $f_a$ tends to zero uniformly on compact subsets of $\D$ as $|a|\rightarrow 1^-$. Then \cite[Lemma 2.1]{PelSum14} implies that we have $\|f_{a}\|_{A^p_\omega}\asymp 1$ for all $a\in \D$ if $\gamma=\gamma(\omega,p)>0$ sufficient large.  By Lemma~\ref{compact}, we have $\|D^n_{\varphi,u}f_{a}\|_{A^q_\nu}\rightarrow 0$ as $|a|\rightarrow 1^-$. Hence

\begin{equation*}
  \begin{aligned}
  &\lim_{|a|\rightarrow 1^-}(C(\lambda,p,n)|a|^n)^q \int_\D|u(z)|^q\frac{(1-|a|)^{\gamma q}}{|1-\overline{a}\varphi(z)|^{(\gamma+n) q}} \frac{\nu(z)}{\omega(S(a))^{\frac{q}{p}}}dA(z)\\
  &=\lim_{|a|\rightarrow 1^-}\|D^n_{\varphi,u}f_{a}\|^q_{A^q_\nu}=0,
  \end{aligned}
\end{equation*}
and \eqref{lim} follows.

Conversely, assume that \eqref{lim} holds. By Proposition~\ref{Carleson d}, it suffices to show that
   $$
   \lim_{|a|\rightarrow 1^-}\frac{\varphi_*(|u|^q\nu)(\Delta(a,r))}{\omega(S(a))^{\frac{q}{p}}(1-|a|)^{nq}}= 0.
   $$
By the hypothesis and the estimate $|1-\overline{a}\varphi(z)|\asymp 1-|a|$ for $z\in \varphi^{-1}\De(a,r)$ we have
\begin{equation}
  \begin{split}
  0
  &=\lim_{|a|\rightarrow 1} \int_\D|u(z)|^q\frac{(1-|a|)^{\gamma q}}{|1-\overline{a}\varphi(z)|^{(\gamma+n) q}}\frac{\nu(z)}{\omega(S(a))^{\frac{q}{p}}}dA(z)\\
   &\gtrsim \lim_{|a|\rightarrow 1}\int_{\varphi^{-1}(\Delta(a,r))} |u(z)|^q\frac{(1-|a|)^{\gamma q}}{|1-\overline{a}\varphi(z)|^{(\gamma+n) q}} \frac{\nu(z)}  
  {\omega(S(a))^{\frac{q}{p}}}dA(z)\\
  &\gtrsim \lim_{|a|\rightarrow 1}\frac{\varphi_*(|u|^q\nu)(\Delta(a,r))}{\omega(S(a))^{\frac{q}{p}}(1-|a|)^{nq}}.
  \end{split}
\end{equation}
Thus $D^n_{\varphi,u}$ is compact. The proof is complete.
\end{proof}

Now we consider the compactness of the operator $D^n_{\varphi,u}$ acting from $A^p_\omega$ to $H^\infty(\D)$. 

\begin{proposition} \label{compact p}
Let $\omega \in \DDD$, $0<p<\infty$ and $n \in \N \cup\{0\}$. Let $\varphi$ be an analytic self-map of $\D$ and $u\in \H^\infty(\D)$. Then the following statements are equivalent:
\begin{itemize}
\item[(i)] $D^n_{\varphi,u}: A^p_\omega \rightarrow H^\infty$ is compact;
\item[(ii)] $\sup_{z\in \D}|\varphi(z)|<1$ or 
  \begin{equation} \label{condition2}
   \lim_{|\varphi(z)|\rightarrow 1} \frac{|u(z)|}{\omega(S(\varphi(z)))^{\frac{1}{p}}(1-|\varphi(z)|)^n}=0.
  \end{equation}
\end{itemize}
\end{proposition}

\begin{proof}
First, we show that (i) implies (ii). Assume that $\sup_{\eta\in \D}|\varphi(\eta)|=1$ and 
   \begin{equation*}
    \lim_{|\varphi(\eta)|\rightarrow 1} \frac{|u(\eta)|}{\omega(S(\varphi(\eta)))^{\frac{1}{p}}(1-|\varphi(\eta)|)^n}\neq 0.
   \end{equation*}
Then there exist $\varepsilon >0$ and a sequence $\{\eta_k\}$ in $\D$ such that $\lim_{k\rightarrow\infty}|\varphi(\eta_k)|=1$  and 
  \begin{equation*}
   \frac{|u(\eta_k)|}{\omega(S(\varphi(\eta_k)))^{\frac{1}{p}}(1-|\varphi(\eta_k)|)^n}> \varepsilon
  \end{equation*}
for all $k\in \N$. Consider the function $f_{\eta_k}$ with $a=\varphi(\eta_k)$ defined in \eqref{eq:testfunctions}. Then $f_{\eta_k} \in A^p_\omega$, $\|f_{\eta_k}\|_{A^p_\omega}\leq C$ and $f_{\eta_k} \rightarrow 0$ uniformly on compact subsets of $\D$ as $|\varphi({\eta_k})|\rightarrow 1^-$. Since $D^n_{\varphi,u}$ is compact, we have $\|D^n_{\varphi,u} f_{\eta_k}\|_\infty \rightarrow 0$ as $|\varphi({\eta_k})|\rightarrow 1^-$ by Lemma \ref{compact}. On the other hand 

  \begin{equation}
  \|D^n_{\varphi,u} f_{\eta_k}\|_\infty = \sup_{z\in\D}|u(z)||f^{(n)}_{\eta_k}(\varphi(z))|\gtrsim \frac{|u({\eta_k})|}{\omega(S(\varphi({\eta_k})))^{\frac{1}{p}}(1-|\varphi(\eta_k)|)^n}> \varepsilon.
  \end{equation}
This is a contradiction.

Next, we show that (ii) implies (i). It suffices to show that for any norm bounded sequence $\{f_j\}$ in $A^p_\omega$ which tends to zero uniformly on compact subsets of $\D$ as $j\rightarrow \infty$, we have $\|D^n_{\varphi,u}f_j\|_\infty \rightarrow 0$ as $j\rightarrow \infty$.  First, we assume that $\sup_{z\in \D}|\varphi(z)|<1$ and \eqref{condition2} holds. \eqref{condition2} and $u\in \H^\infty(\D)$ imply
  \begin{equation}
   \sup_{z\in \D} \frac{|u(z)|}{\omega(S(\varphi(z)))^{\frac{1}{p}}(1-|\varphi(z)|)^n}< \infty.
  \end{equation}
By Proposition~\ref{bounded p} it follows that $D^n_{\varphi,u}: A^p_\omega \rightarrow \H^\infty$ is bounded. Therefore, for any $\varepsilon > 0$, there exists $r \in (0,1)$ such that
   \begin{equation} 
   \frac{|u(z)|}{\omega(S(\varphi(z)))^{\frac{1}{p}}(1-|\varphi(z)|)^n}< \varepsilon,
   \end{equation}
provided $r<|\varphi(z)|<1$. For $z\in \D$ with $r<|\varphi(z)|<1$, we therefore have 
   \begin{equation}
    |u(z)||f^{(n)}_j(\varphi(z))|\leq \frac{|u(z)|\|f_j\|_{A^p_\omega}}{(\omega(S(\varphi(z))))^{\frac{1}{p}}(1-|\varphi(z)|)^n}\leq C \varepsilon.
   \end{equation}
By Cauchy's estimate, we see that $f^{(n)}_j\rightarrow 0$ uniformly on compact subsets of $\D$ as $j\rightarrow \infty$. Hence there exists a  $j_0=j_0(\varepsilon) \in \N$ such that $|f^{(n)}_j(\varphi(z))|< \varepsilon$ for all $j\geq j_0$ and $|\varphi(z)|\leq r$. Therefore, for all $j\geq j_0$, we have 
    \begin{equation*}
    \lim_{j\rightarrow \infty}\|D^n_{\varphi,u} f_j\|_\infty=\lim_{j\rightarrow \infty} \sup_{z\in \D} |u(z)|f_j^{(n)}(\varphi(z))| =0.
    \end{equation*}
This finishes the proof.
\end{proof}

\section{PROOFS OF MAIN RESULTS} \label{Sec:proof-of-main}
With the propositions proved in the previous two sections we can now easily obtain the main results stated in the introduction.

\medskip
\noindent\emph{Proof of} Theorem~\ref{Carleson DD}. By the measure theoretic change of variable \cite[Section 39]{PM}, it follows that $\|D^n_{\varphi,u}f\|^q_{L^q_\nu} = \|f^{(n)}\|^q_{L^q_{\varphi_*(|u|^q\nu)}}$ for each $f\in A^p_\omega$. Therefore the theorem is an immediate consequence of Proposition~\ref{Carleson d2}.
\hfill$\Box$

\medskip
\noindent\emph{Proof of} Theorem~\ref{Carlson D}. This result follows by Propositions~\ref{Carleson d} and~\ref{Carleson d2}. Namely, these propositions characterize the $q$-Carleson measures for $A^p_\omega$ in the cases $0<p\leq q <\infty$ and $0<q<p<\infty$, respectively.   
\hfill$\Box$

\medskip
\noindent\emph{Proof of} Theorems~\ref{theorem1} and~\ref{Th}. Theorems~\ref{theorem1} and~\ref{Th} are immediate consequences of Propositions~\ref{bouned pq} and \ref{compact pq} and Propositions~\ref{bounded p} and \ref{compact p}, respectively.
\hfill$\Box$

\section*{Acknowledgement}
The author would like to thank Prof. Jouni R\"atty\"a for his valuable advice which helped to clarify the whole paper. The author is very grateful to the  referee for a detailed reading of the paper and recommendations.

\end{document}